\documentclass{article}

\usepackage[T2A]{fontenc}
\usepackage[cp1251]{inputenc}
\usepackage[english,russian]{babel}
\usepackage[tbtags]{amsmath}
\usepackage{amsfonts,amssymb}

\usepackage{mathrsfs}

\usepackage[paper]{mz2ewin}

\usepackage{graphicx}

\usepackage{array}
\usepackage{booktabs}
\usepackage[hidelinks]{hyperref}
\numberwithin{equation}{section}

\theoremstyle{plain}

\newtheorem{theorem}{Theorem}
\newtheorem{lemma}{Lemma}

\newtheorem{proof}{Proof}

\begin{document}

%\udk{512.542.1, 512.54.05, 512.542.7}

%\date{\\
%Исправленный вариант\\ }

\author{G.\,V.~Antiufeev}
\address{Independent scientist}
\email{grigoriy.rus@gmail.com}

\title{A Lower Bound for the Diameter of Cayley Graph of the Symmetric Group $S_n$ Generated by $(12), (12 \dots n), (1n \dots 2)$}

\markboth{}{}

\maketitle

\begin{fulltext}

\begin{abstract}
Let us denote elements of the symmetric group $S_n$ using square brackets for one-line notation. Cycles will be represented using parentheses, following standard notation.
Under this convention, the full reversal of the identity element $()$ is the element $s = [n\ n-1 \dots 1]$.

In this article, we obtain a lower bound on the decomposition complexity of elements $s(1n \dots 2)^{i}$ into the generators $(12), (12 \dots n), (1n \dots 2)$, where $i$ ranges over the set $\{1,2,\dots,n\}$.
As a consequence, we derive the lower bound $n(n-1)/2$ for the diameter of Cayley graph of the group $S_n$ generated by $(12), (12 \dots n), (1n \dots 2)$.
\end{abstract}


\section{Introduction}

Let $G$ be a group and let $J$ be a subset of $G$. We say that $J$ generates $G$, or that $J$ is a generating set of $G$, if every element of $G$ can be expressed as a product of elements of $J$ and their inverses. 
The structure of a group is encoded by its Cayley graph, whose vertices are the elements of the group. Multiplying each element by the generators produces adjacent vertices connected by edges.  
In the standard setting, all edges are assigned weight equal to one.
In this case, the diameter of the graph is equal to the maximum length of the shortest word representing an element of the group in terms of the given generators and their inverses.  
Thus, the diameter reflects the worst-case algorithmic complexity of expressing a group element as a product of generators and their inverses.

The symmetric group $S_n$ is of particular interest since, according to Cayley's theorem, any finite group is isomorphic to a subgroup of some symmetric group.  
The elements of $S_n$ are permutations, which connects the study of Cayley graphs of this group with sorting problems \cite{Zubov2016}.

One of the minimal generating sets for $S_n$ consists of the two elements $(12)$ and $(12 \dots n)$. This set is also important because the Cayley graph for these generators contains a Hamiltonian path \cite{SawadaWilliams19}, which permits the construction of a Gray code.
Different weights for the edges corresponding to the generators affect the graph’s diameter.
For instance, if the weight of all edges corresponding to the element $(1n \dots 2)$ is assigned to~$\infty$, diameter is asymptotic to $\frac{3n^2}{4}$.
This result was obtained in~\cite{Zubov}.  
If the weight of all edges corresponding to $(12 \dots n)$ and $(1n \dots 2)$ are assigned to zero, the diameter equals $\left\lfloor \frac{n}{2} \right\rfloor \left\lfloor \frac{n-1}{2} \right\rfloor$, as shown in \cite{Zubov2016}.
This result was later independently obtained in~\cite{Alon}.

In the present work, we assume that all edges of the graph have weight equal to one. 
An asymptotic estimate of order $\Theta(n^2)$ for the diameter of this graph was obtained in~\cite{Babai}. Later, an upper bound of $\frac{3}{2}n^2$ was established in~\cite{Kuppili20}. 
A further improvement of these bounds was obtained in~\cite{ChervovRL}: the lower bound is equal to $\frac{n(n-1)}{2}-\frac{n}{2}-1$ and the upper bound is equal to $\frac{n(n-1)}{2}+3n$.

We denote elements of $S_n$ using square brackets for one-line notation.  
Cycle notation is represented using parentheses.  
Thus, the full reversal of the identity element $()$ of $S_n$ yields the permutation $s = [n\ n-1 \dots 1]$.  
In \cite{ChervovRL}, it was hypothesized that the most difficult element, for which the lower bound of the diameter of the Cayley graph of the symmetric group $S_n$ with generators $(12), (12 \dots n), (1n \dots 2)$, equal to $\frac{n(n-1)}{2}$, is attained, is the permutation $s$ shifted by two, i.e., the element $s(1n \dots 2)^{2}$.

In this paper we obtain a lower bound on the decomposition complexity of elements $s(1n \dots 2)^{i}$ into the generators $(12), (12 \dots n), (1n \dots 2)$, where $i$ ranges over the set $\{1,2,\dots,n\}$.
As a consequence, we derive the lower bound $\frac{n(n-1)}{2}$ for the diameter of Cayley graph of the group $S_n$ generated by $(12), (12 \dots n), (1n \dots 2)$.

\section{Definitions}
\label{sec::def}

We introduce several definitions. Additional definitions can be found, for example, in~\cite{Konstantinova} and \cite{Zhuravlev_Flerov_Vyaliy}.

\textit{A permutation} of length $n$ of elements of the set $\mathbb{Z}_n$ (with $n$ used as a representative of $\overline{0}$) is given by:
$$
\pi=
\bigl(\begin{smallmatrix}
  1 & 2 & 3 & \dots & n-1 & n \\
  \sigma(1) & \sigma(2) & \sigma(3) & \dots & \sigma(n-1) & \sigma(n)
\end{smallmatrix}\bigr),
$$ 
where $\sigma(i)$ denotes the image of the element $i$, $\sigma(i),i \in \mathbb{Z}_n$.
We also use a more compact notation for such permutations, employing square brackets in order to avoid confusion with cycle notation: 
$$
\pi = 
[\ \pi_1\ \ \pi_2\ \ \pi_3\ \ \dots\ \ \pi_{n-1}\ \ \pi_n\ ]
=
[\ \sigma(1)\ \ \sigma(2)\ \ \sigma(3)\ \ \dots\ \ \sigma(n\!-\!1)\ \ \sigma(n)\ ].
$$
In this notation, unlike the classical two-line representation, the order of the entries is essential.
That is, the elements of a permutation may be regarded as forming a sequence.

Let $K$ be a set equipped with a strict linear order $\prec$ (we will omit the word ``strict'' in the following).
Elements $a,b,c \in K$ are said to be in the relation of a \textit{cyclic order $\prec_{cycle}$} if and only if one of the following conditions holds \cite[p.6]{Hunt1935}:
$$
a \prec b \prec c, \quad b \prec c \prec a, \quad c \prec a \prec b.
$$
In this case, we say that the linear order $\prec$ \textit{induces the cyclic order $\prec_{cycle}$}.
An element $z \in K$ \textit{determines a linear order $\prec_z$} on $K \setminus \{z\}$ as follows: $a \prec_z b$ if and only if the elements $a,b,z$ are in the cyclic order~$\prec_{cycle}$ \cite[p.7]{Hunt1935}.
%Below, after the introduction of the notion of orbits, an important remark on the preservation of cyclic order will be given.

\textit{The group of all permutations of length
$n$} is denoted by $S_n$, and \textit{the identity element} of the group $S_n$ is denoted by $()$. 

\textit{The dihedral group} $
D_n \;=\; \langle R,S \mid R^n = e,\; S^2 = e,\; SRS = R^{-1} \rangle.$
The identity element of $D_n$ is $e$.
The group $D_n$ is interpreted as the group of symmetries of a regular $n$-gon: the element $R$ corresponds to a rotation by an angle $\frac{2\pi}{n}$, while the element $S$ corresponds to a reflection with respect to a symmetry axis.
The group $D_n$ contains $2n$ elements: $e=R^n$, $n-1$ rotations $R^i$, and $n$ reflections $SR^j$, where $i = 1,2,\dots,n-1$ and $j = 1,2,\dots,n$.
Each element $SR^j$ represents a reflection about the axis obtained by rotating the original reflection axis (corresponding to $S$) by the angle $\frac{2\pi j}{n}$.

We construct a subgroup $H_n < S_n$ such that $D_n \cong H_n$, for $n \geq 3$. Let
$
r = [2\ 3\ \dots\ n\ 1] = (12 \dots n), \quad s = [n\ n-1 \dots\ 1] = (1n \dots 2).
$
The permutation $r$ is called \textit{a left shift}, the permutation $r^{-1}$ is called \textit{a right shift}, and the permutations $sr^{j}$, where $j = 1,2,\dots,n$, will be called \textit{reflections}.
Then $H_n \;=\; \langle r,s \mid r^n = (),\; s^2 = (),\; srs = r^{-1} \rangle$ and the isomorphism is given by the mapping $\varphi: D_n \longrightarrow H_n$, defined by:
$$\varphi(R^i)=r^i,\quad
\varphi(SR^i)=sr^i,\quad
\varphi^{-1}(r^i)=R^i,\quad
\varphi^{-1}(sr^i)=SR^i.
$$
Thus, we will work with elements of the subgroup $H_n$ as with elements of the dihedral group.

We say that a permutation $\alpha$ is \textit{conjugate to a permutation $\beta$ via $\gamma$} if
$
\alpha = \gamma \beta \gamma^{-1},
$
where $\alpha,\beta,\gamma \in S_n$.

Let the cycle $(12)$ be denoted by $\delta$, $\delta \in S_n$. That is, $\delta = [2\ 1\ 3\ 4\ \dots\ n\!-\!1\ n].$

The Cayley graph $\Gamma = (V,E)$ of the symmetric group $S_n$ with generators $\delta, r, r^{-1}$ is defined as the graph with vertex set $V = S_n$ and edge set
\[
E = \{\{g,h\} : g,h \in S_n,\; g^{-1}h \in \{\delta, r, r^{-1}\}\}.
\]

\textit{The distance} $\textup{dist}(\pi',\pi'')$ between two permutations $\pi'$ and $\pi''$, where $\pi',\pi'' \in S_n$, is defined as the length of the shortest path in $\Gamma$ connecting the vertices corresponding to $\pi'$ and $\pi''$. 
The pair $S_n$ and $\textup{dist}(\cdot):S_n \times S_n \rightarrow \mathbb{R}$ forms a metric space.

The greatest distance between any two permutations is \textit{the diameter} of $\Gamma$:
$$
\textup{diam}(\Gamma) = 
\max_{
\pi,\xi \in S_n
}
\textup{dist}(\pi,\xi).
$$

\textit{The orbit} of the action of the group $\langle r \rangle$ on a permutation $\pi \in S_n$: $$\textup{Orb}(\pi) = \{\pi g \mid g \in \langle r \rangle\}.$$

Let $\pi \in S_n$ and $\xi \in$~$S_n$.
\textit{The distance} between a permutation
$\pi$ and $\textup{Orb}(\xi)$ is defined as:
$$
\textup{dist}(\pi,\textup{Orb}(\xi)) = 
\min_{
\substack{\xi' \in \textup{Orb}(\xi)}
}
\textup{dist}(\pi,\xi').
$$

Let $\pi = [\pi_1 \dots \pi_{i-1}\pi_i\pi_{i+1}\dots\pi_{j-1}\pi_{j}\pi_{j+1}\dots\pi_n]
$.
Then, any permutation in $\textup{Orb}(\xi)$, where
$$\xi = [\pi_1 \dots \pi_{i-1}\pi_{j}\pi_{j-1} \dots \pi_{i+1}\pi_{i}\pi_{j+1}   \dots\pi_n],
$$
which reverses the order of elements within the subsequence $\pi_i,\pi_{i+1},\dots,\pi_j$, is called a \textit{$(\pi_i,\pi_j)$-reversal of length $j-i+1$ of the permutation $\pi$}.
Sometimes we simply use the terms \textit{reversal} or \textit{reversal of length $j-i+1$}.
The operation of moving from a permutation to its $(\pi_i,\pi_j)$-reversal is called \textit{reversing} the subsequence $\pi_i,\pi_{i+1},\dots,\pi_j$.
%, or, more concisely, the subsequence $(\pi_i,\pi_j)$.
Thus, reversing the subsequence $\pi_1,\pi_2,\dots,\pi_n$ of a permutation $\pi$ is equivalent to multiplying by $sr^x$, where $x \in \mathbb{Z}_n$.

\section{A Lower Bound for the Diameter of Cayley Graph of the Symmetric Group $S_n$ Generated by $(12), (12 \dots n), (1n \dots 2)$}

Before proceeding to the proof of the main results of this paper, we determine the distances from a permutation to its nearest reversals of length not exceeding $\left\lceil\frac{n}{2}\right\rceil$.
In analogy with classical sorting problems, a reversal can be viewed as sorting a decreasing subsequence of elements, which corresponds to the worst-case complexity of the sorting process. That is, every pair of elements in the subsequence must be sorted. By the definition given above, a reversal is any permutation from the orbit of the original permutation in which the subsequence is sorted. Thus, we are dealing with sorting on a cyclic structure. From an algorithmic point of view, the particular permutation obtained from the orbit depends on the sequence of steps performed.
In classical sorting problems, adjacent transpositions are used. In our setting, there is only one such transposition, namely $\delta = (12)$. Therefore, it is necessary to use shifts, which are also generators in this problem, in order to swap arbitrary adjacent elements of the permutation.
Thus, in Lemma~\ref{lemma}, we determine the distance to the nearest reversals belonging to the orbit of permutations in which a subsequence of length at most $\left\lceil\frac{n}{2}\right\rceil$ is reversed. There are four such nearest reversals, corresponding to formulas~\ref{lem_1}--\ref{lem_4}. In algorithmic terms, the decompositions used in the proofs of formulas~\ref{lem_1}--\ref{lem_4} can be described as follows.
The decomposition used for formula~\ref{lem_1} corresponds to sorting the subsequence from left to right, starting with the first two elements.
For formula~\ref{lem_2}, the decomposition corresponds to sorting from right to left, starting with the last two elements.
For formula~\ref{lem_3}, the decomposition corresponds to sorting from left to right, starting from the middle of the subsequence. The decomposition for formula~\ref{lem_4} corresponds to sorting from right to left, starting from the middle.
We now proceed to the statement and proof of Lemma~\ref{lemma}.

\begin{lemma}\label{lemma}
Let 
$$
\pi = [\pi_1 \dots \pi_{j-1}\pi_j\pi_{j+1}\dots\pi_n],
$$
$$\xi = [\pi_{j}\pi_{j-1} \dots \pi_{1}\pi_{j+1} \dots \pi_n],
$$
where $\pi, \xi \in S_n, 2 \leqslant j \leqslant \left\lceil\frac{n}{2}\right\rceil, n \geqslant 3.$
Then 
\begin{align}
\label{lem_1}
\textup{dist}(\pi,\xi r^{\left\lfloor\frac{j}{2}\right\rfloor-1}) &= j(j-1)-1, \tag{$\textup{I}$} \\
\label{lem_2}
\textup{dist}(\pi r^{j-2},\xi r^{\left\lceil\frac{j}{2}\right\rceil-1}) &= j(j-1)-1, \tag{$\textup{II}$} \\
\label{lem_3}
\textup{dist}(\pi r^{\left\lfloor\frac{j}{2}\right\rfloor-1},\xi) &= j(j-1)-1, \tag{$\textup{III}$} \\
\label{lem_4}
\textup{dist}(\pi r^{\left\lceil\frac{j}{2}\right\rceil-1},\xi r^{j-2}) &= j(j-1)-1. \tag{$\textup{IV}$}
\end{align}
\end{lemma}

\begin{proof}
\normalfont
We prove \textbf{Formula \ref{lem_1}}.
A permutation $\xi$ is a reversal of length $j$ of the permutation $\pi$, in which the elements $\pi_k$ and $\pi_l$ are swapped, with $1 \leqslant k \leqslant \left\lfloor\frac{j}{2}\right\rfloor$, $l = j + 1 - k$, $k < l$.

The linear order $<$ defined on the set of indices $\mathbb{Z}_n$ naturally induces a linear order~$\prec$ on the elements of the permutation $\pi$, which, in turn, induces the cyclic order $\prec_{cycle}$. That is, one may say that all elements of the permutation $\pi$ are in the cyclic order~$\prec_{cycle}$.
An \textit{inversion} is defined as a pair of elements $\pi_p$ and $\pi_q$ such that $\pi_p \nprec_{\pi_{j'}} \pi_q$, where $p,q \in \{1,2,\dots,j\}$, $p < q$, $j' \in \{j+1,j+2,\dots,n\}$, and where $\prec_{\pi_{j'}}$ is the linear order determined by the element $\pi_{j'}$.
The number of inversions of the sequence $\pi_1 \dots \pi_{j-1} \pi_j$ is equal to $0$, whereas the number of inversions of the sequence $\pi_j \pi_{j-1} \dots \pi_1$ is equal to $\frac{j(j-1)(n-j)}{2}$.

Consider the permutation $\pi r^{k-1}$, whose first position is occupied by the element $\pi_k$. To optimally swap the elements $\pi_k$ and $\pi_l$ we use the following decomposition:
\begin{equation}
\label{reversal_1}
\pi r^{k-1}\left(\delta r\right)^{l-k-1}\delta\left(r^{-1}\delta\right)^{l-k-1} \tag{$\mathcal{A}$}.
\end{equation}

Optimality is ensured by the following considerations.
Consider a permutation $\kappa \in$~$ S_n$. If $\kappa_1,\kappa_2 \in \{\pi_1,\pi_2,\dots,\pi_j\}$, then multiplication of $\kappa$ by $\delta$ changes the number of inversions by $n-j$.
If $\kappa_1,\kappa_2 \in \{\pi_{j+1},\pi_{j+2},\dots,\pi_n\}$, then multiplication of $\kappa$ by $\delta$ does not change the number of inversions.
If
$
\kappa_1 \in \{\pi_1,\pi_2,\dots,\pi_j\}, 
\kappa_2 \in \{\pi_{j+1},\pi_{j+2},\dots,\pi_n\},
$
or
$
\kappa_1 \in$~$ \{\pi_{j+1},\pi_{j+2},\dots,\pi_n\}, 
\kappa_2 \in \{\pi_1,\pi_2,\dots,\pi_j\},
$
then multiplication of $\kappa$ by $\delta$ changes the number of inversions by $j-1$.
Thus, each multiplication by $\delta$ in decomposition~\ref{reversal_1} increases the number of inversions by $n-j$.
Since $\delta^2 = ()$, there must be at least one shift between consecutive applications of $\delta$, which is indeed present in the decomposition.
Clearly, a shift preserves the number of inversions.
The number of group elements in decomposition~\ref{reversal_1} is equal to $4(l-k-1)+1$.

To obtain $\xi r^{\left\lfloor\frac{j}{2}\right\rfloor-1}$ one needs to swap $\left\lfloor\frac{j}{2}\right\rfloor$ pairs of elements of the permutation. The minimal number of shifts required to move from one pair of elements to the next is $\left\lfloor\frac{j}{2}\right\rfloor\!-\!1$.
Again, we note that shifts preserve the number of inversions.
Therefore, the total number of permutations, required to swap $\left\lfloor\frac{j}{2}\right\rfloor$ pairs using decomposition \ref{reversal_1} is:
\[
\textup{dist}(\pi,\xi r^{\left\lfloor\frac{j}{2}\right\rfloor-1}) = \sum_{k=1}^{\left\lfloor\frac{j}{2}\right\rfloor} \bigl(4(l - k - 1) + 1\bigr) + \left\lfloor\frac{j}{2}\right\rfloor - 1,
\quad \text{where } l = j + 1 - k.
\]

To obtain the desired value, substitute $l = j + 1 - k$, simplify the expression under the summation, compute the sum of the first $\left\lfloor\frac{j}{2}\right\rfloor$ terms of the resulting progression, and consider separately the cases of even and odd $j$.

The total number of permutations $\delta$ when using decomposition~\ref{reversal_1} is
\[
\sum_{k=1}^{\left\lfloor\frac{j}{2}\right\rfloor} \bigl(2(l - k - 1) + 1\bigr) = \frac{j(j-1)}{2},
\quad \text{where } l = j + 1 - k.
\]

Each multiplication by $\delta$ increases the number of inversions by $n-j$. Therefore, the total number of inversions increases by $\frac{j(j-1)(n-j)}{2}$, which coincides with the number of inversions of the sequence $\pi_{j}\pi_{j-1} \dots \pi_{1}$.

Now we prove \textbf{Formula \ref{lem_2}}, that is, $\textup{dist}(\pi r^{j-2},\xi r^{\left\lceil\frac{j}{2}\right\rceil-1}) = j(j-1)-1$.
Consider the permutation $\pi r^{l-2}$, whose first position is occupied by the element $\pi_{l-1}$. We obtain the following decomposition to swap the elements $\pi_k$ and $\pi_l$:
\begin{equation}
\label{reversal_4}
\pi r^{l-2}\left(\delta r^{-1}\right)^{l-k-1}
\delta
\left(r\delta\right)^{l-k-1}. \tag{$\mathcal{B}$}
\end{equation}

Decomposition~\ref{reversal_4} is optimal for the same reasons as decomposition~\ref{reversal_1} and consists of $4(l-k-1)+1$ permutations.
Therefore, the same total number of permutations is obtained for swapping $\left\lfloor\frac{j}{2}\right\rfloor$ pairs of elements as in the previous case.
\textbf{Formulas \ref{lem_3}} and \textbf{\ref{lem_4}} are proved similarly.
For Formula \ref{lem_3}, i.e., $\textup{dist}(\pi r^{\left\lfloor\frac{j}{2}\right\rfloor-1},\xi) = j(j-1)-1$, one should use decomposition \ref{reversal_1} to swap the elements of the permutation.
For Formula \ref{lem_4}, i.e., $\textup{dist}(\pi r^{\left\lceil\frac{j}{2}\right\rceil-1},\xi r^{j-2}) = j(j-1)-1$, one should use decomposition \ref{reversal_4}, to swap the elements of the permutation.
This completes the proof of the lemma.

\end{proof}

\begin{theorem}\label{th_1}
Let $s$ $=$ $[n\ n\!-\!1\ \dots\ 1]$, $r = [2\ 3\ \dots\ n\ 1]$, $s,r \in S_n$, $n \geqslant 4$.
The following bound is valid:
$$
\textup{dist}(sr^{n-i},()) \! = \! \left\lceil\frac{n}{2}\right\rceil \left(\left\lceil\frac{n}{2}\right\rceil\!-\!1\right)-1 + \left\lfloor\frac{n}{2}\right\rfloor\left(\left\lfloor\frac{n}{2}\right\rfloor\!-\!1\right)-1 +
\begin{cases}
\left\lfloor\frac{n}{2}\right\rfloor + 1, & i=1,\\[6pt]
\left\lfloor\frac{n}{2}\right\rfloor - i + 4, & 2 \le \!i\! \le \left\lfloor\frac{n}{2}\right\rfloor\!+\!2,\\[6pt]
i - \left\lceil\frac{n}{2}\right\rceil, & \left\lceil\frac{n}{2}\right\rceil\!+\!2 \le \!i\! \le n.
\end{cases}
$$
\end{theorem}
\begin{proof}
\normalfont

Since $sr^{n-i}sr^{n-i} = ()$, it is necessary to find the decomposition of the permutation $sr^{n-i}$. For simplicity, we first find a decomposition for $s$.
 
The linear order $<$ on the index set $\mathbb{Z}_n$ naturally induces a linear order $\prec$ on the permutation $s$, which in turn gives rise to a cyclic order on $s$. We denote this order by $\prec_{\textup{cycle}}$.
In other words, all elements of the permutation $s$ are in cyclic order with respect to $\prec_{\textup{cycle}}$.
If the elements of a permutation are indexed in the reverse order, one analogously obtains a cyclic order $\succ_{\textup{cycle}}$ on that permutation.
By the definition of cyclic order, this relation is preserved on $\textup{Orb}(\varsigma)$.

We introduce a metric that allows us to compute distances between permutations within a single orbit.
Let $\varsigma \in S_n$ be a permutation. The metric $\textup{dist}(\cdot)$ induces the \textit{Lee metric}~\cite[p.~52]{Dist} on $\textup{Orb}(\varsigma)$:
$$
\textup{dist}(\varsigma r^{i'},\varsigma r^{j'}) = \min\bigl(\, |i'-j'|,\ |n-(i'-j')| \,\bigr),
$$
where $i',j'\in \mathbb{Z}_n$.

This follows from the following considerations. As in the case of the permutation $s$, one can introduce a cyclic order $\prec'_{\textup{cycle}}$ on the elements of $\varsigma$. The shift $r$ preserves this order by definition. On the other hand, the permutation $\delta$ changes the cyclic order of some triple of elements. To restore the original order of these elements, one needs to transpose two of them, which requires another application of $\delta$. Thus, using $\delta$ does not reduce the length of the decomposition.
Thus, the number of shifts required to obtain $\varsigma r^{j'}$ from $\varsigma r^{i'}$ is
$
\min\bigl(\, |i'-j'|,\ |n-(i'-j')| \,\bigr).
$

We now state a fact on transpositions and shifts in a permutation. 
It is well known that any permutation $\varsigma \in S_n$ can be expressed as a product of transpositions, that is,
\[
\varsigma = \sigma_1 \sigma_2 \dots \sigma_k,
\quad \text{where } \sigma_1, \sigma_2, \dots, \sigma_k \in S_n.
\]
We assume that $k$ is minimal.
Let $\varsigma \neq r^p$, $p\in \mathbb{Z}_n$. Consider how the number of transpositions and shifts changes under multiplication of $\varsigma$ by $r^q$, $q\in \mathbb{Z}_n$
(that is, after multiplication by $q$ shifts):
\[
\varsigma r^{q} = ()\varsigma r^{q} = r^{q}r^{-q}\varsigma r^{q} = r^{q}(r^{-q}\varsigma r^{q})
\]
\[
= r^{q}(r^{-q}\sigma_1 \sigma_2 \dots \sigma_k r^{q})
= r^{q}(r^{-q}\sigma_1 () \sigma_2 () \dots () \sigma_k r^{q})
\]
\[
= r^{q}(r^{-q}\sigma_1 r^q r^{-q} \sigma_2 r^q r^{-q} \dots r^q r^{-q} \sigma_k r^{q})
\]
\[
= r^{q}\bigl((r^{-q}\sigma_1 r^{q})(r^{-q} \sigma_2 r^{q}) \dots (r^{-q} \sigma_k r^{q})\bigr).
\]

Thus, multiplication by $r^q$ on either the left or the right preserves the number of transpositions in the permutation $\varsigma$, conjugating them.
The number of shifts $q$ is also preserved.

We now proceed to show that the elements $s_1, s_2, \dots, s_n$ of the permutation $s$ are ordered with respect to the cyclic order relation $\succ_{\textup{cycle}}$ on the permutation $()$.
The element $s_p$, where $p \in \mathbb{Z}_n$, coincides with the element $()_{n-p+1}$ of the permutation $()$.
Let us choose any three elements $s_k, s_l, s_m$, where $k,l,m \in \mathbb{Z}_n$, of the permutation $s$. Without loss of generality, assume that $k < l < m$.
In the permutation $()$, these elements correspond respectively to $()_{n-k+1}$, $()_{n-l+1}$, and $()_{n-m+1}$.
Since $n-k+1 > n-l+1 > n-m+1$, the elements $s_k, s_l, s_m$ are ordered with respect to the cyclic order relation $\succ_{\textup{cycle}}$ on the permutation $()$.
It follows that the elements $s_1, s_2, \dots, s_n$ are ordered with respect to the cyclic order relation $\succ_{\textup{cycle}}$ on the permutation $()$.
It is clear that the elements $s_1, s_2, \dots, s_n$ are ordered with respect to the cyclic order relation $\succ_{\textup{cycle}}$ only on permutations belonging to $\textup{Orb}(())$.
This follows directly from the definition of the cyclic order and from the fact that the action of the group $\langle r \rangle$ sends indices of any triple of elements to cyclically equivalent ones.
Therefore, in order for the elements $s_1, s_2, \dots, s_n$ to be ordered with respect to the cyclic order relation $\succ_{\textup{cycle}}$, one must multiply $s$ by a number of permutations equal to $\textup{dist}(s, \textup{Orb}(()))$.
This is equivalent to reflecting the permutation $s$ with respect to some axis of symmetry.

The reflection of $s$, viewed as an element of the dihedral group, with respect to an axis of symmetry of a regular $n$-gon, is a product of transpositions corresponding to pairwise reflections of the vertices of the regular $n$-gon.
These transpositions naturally split into two sets $\textup{F}_1$ and $\textup{F}_2$ such that $|\textup{F}_1| = \left\lfloor\frac{n}{2}\right\rfloor$ and $|\textup{F}_2| = \left\lceil\frac{n}{2}\right\rceil$. 
These sets can be chosen so that they correspond to subsequences of the original permutation.

Reversing these two subsequences realizes the reflection by definition. For an optimal reversal, one uses Lemma~\ref{lemma}, which gives a distance between a permutation and its reversal of length at most $\left\lceil\frac{n}{2}\right\rceil$.
Since $\textup{F}_1 \cap \textup{F}_2 = \varnothing$, the reversals of these subsequences are independent. Therefore, the minimal number of permutations required for the simultaneous reversal of $\textup{F}_1$ and $\textup{F}_2$ is equal to the sum of the minimal numbers of permutations required for each subsequence.

In the proof of Lemma~\ref{lemma}, a decomposition in terms of the generators $\delta$, $r$, and $r^{-1}$ is used. Thus, the transpositions corresponding to pairwise reflections of the vertices of the regular $n$-gon decompose into the transpositions $\delta = (12)$ and shifts.

Since a reflection with respect to an axis of symmetry yields an element of $\textup{Orb}(())$, additional permutations may be required in the decomposition of the element $s$ in order to obtain the permutation $()$ in the end.
Adding further applications of $\delta$ would destroy the optimality of the decompositions used in the proof of Lemma~\ref{lemma}. Hence, only shifts may appear as additional permutations, and their number must be determined.

We now use the fact on the number of transpositions and shifts in an arbitrary permutation stated above. It follows that the decomposition consisting of an optimal reflection of the permutation $s$ with respect to an axis of symmetry, followed by a minimal number of additional shifts required to obtain $()$, is optimal.

Thus, let the first subsequence $\textup{F}_1$ have length $\left\lfloor\frac{n}{2}\right\rfloor$:
\[
n+m,\, n+m-1,\, \dots,\, n,\, \dots,\, n+m-\left\lfloor\frac{n}{2}\right\rfloor+1,
\]
where $0 \leqslant m \leqslant n-1$.

The second subsequence $\textup{F}_2$ has length $\left\lceil\frac{n}{2}\right\rceil$:
\[
n+m-\left\lfloor\frac{n}{2}\right\rfloor,\, n+m-\left\lfloor\frac{n}{2}\right\rfloor-1,\, \dots,\, 1,\, n,\, n-1,\, \dots,\, n+m+1.
\]

Such a decomposition is equivalent to the following procedure. The indices of all elements of the sequence $s = [s_1\ s_2\ \dots\ s_n]$ are decreased by $m$, i.e., we apply a right shift $sr^{-m}$. The first $\left\lfloor\frac{n}{2}\right\rfloor$ elements of the permutation $sr^{-m}$ form $\textup{F}_1$, while the remaining elements form $\textup{F}_2$.

For the first application of Lemma \ref{lemma}, we assume $j=\left\lfloor\frac{n}{2}\right\rfloor$ and first consider those values of $m$ for which Lemma \ref{lemma} determines the distances between a permutation $s$ and some $(n+m,n+m-\left\lfloor\frac{n}{2}\right\rfloor+1)$-reversal of the permutation $s$.
These values will be: $0,\left\lfloor\frac{n}{2}\right\rfloor-2,\left\lfloor\frac{\left\lfloor\frac{n}{2}\right\rfloor}{2}\right\rfloor-1,\left\lceil\frac{\left\lfloor\frac{n}{2}\right\rfloor}{2}\right\rceil-1$.
Each of these values corresponds to one of the formulas \ref{lem_1}-\ref{lem_4}.
When choosing one of these values of $m$ and the corresponding formula \ref{lem_1}-\ref{lem_4}, the permutation $s$, which is the original one, will be substituted into the distance function as the first argument.
That is, the distance between $s$ and an $(n+m,n+m-\left\lfloor\frac{n}{2}\right\rfloor+1)$-reversal of $s$ will be computed.
If, however, in one of the formulas~\ref{lem_1}-\ref{lem_4} for the reversal of $\textup{F}_1$
we were to substitute not $s$ but a permutation $sr^y$, $y\in\mathbb{Z}_n$, $y\neq n$,
then additional shifts of $s$ would be required.
This would violate the convention that all auxiliary shifts not involved in the decompositions of Lemma~\ref{lemma} are applied only after reflecting the permutation~$s$.

In order to proceed from the reversal of the subsequence $\textup{F}_1$ to the reversal of the subsequence $\textup{F}_2$, it is necessary that at least two elements of $\textup{F}_2$ occupy the first positions of the permutation obtained after the first reversal.
This necessitates at least two shifts, taking into account the Lee metric.
There are only two possible cases in which two shifts will be enough. Let us consider them separately.

Case 1. Let $m = \left\lceil\frac{\left\lfloor\frac{n}{2}\right\rfloor}{2}\right\rceil - 1$. Then by the Lemma \ref{lemma} for $j=\left\lfloor\frac{n}{2}\right\rfloor$ from the Formula \ref{lem_4} we obtain:
$$
\textup{dist}(sr^{-\left(\left\lceil\frac{\left\lfloor\frac{n}{2}\right\rfloor}{2}\right\rceil - 1\right)}r^{\left\lceil\frac{\left\lfloor\frac{n}{2}\right\rfloor}{2}\right\rceil - 1},\xi' r^{\left\lfloor\frac{n}{2}\right\rfloor-2}) =
\textup{dist}(s,\xi' r^{\left\lfloor\frac{n}{2}\right\rfloor-2}) = \left\lfloor\frac{n}{2}\right\rfloor\left(\left\lfloor\frac{n}{2}\right\rfloor-1\right)-1,
$$
where $\xi'$ is the permutation obtained from $sr^{-\left(\left\lceil\frac{\left\lfloor\frac{n}{2}\right\rfloor}{2}\right\rceil - 1\right)}$ by reversing $\textup{F}_1$.

In order to use the Lemma \ref{lemma} to reverse the subsequence $\textup{F}_2$ it is necessary and sufficient to apply two shifts to the element $\xi' r^{\left\lfloor\frac{n}{2}\right\rfloor-2}$.
Thus we get the element $\xi' r^{\left\lfloor\frac{n}{2}\right\rfloor}$.

Let us apply the Lemma \ref{lemma} again for $j=\left\lceil\frac{n}{2}\right\rceil$.
From the Formula \ref{lem_1} we have:
$$
\textup{dist}(\xi' r^{\left\lfloor\frac{n}{2}\right\rfloor},\xi'' r^{\left\lfloor\frac{\left\lceil\frac{n}{2}\right\rceil}{2}\right\rfloor-1}) = \left\lceil\frac{n}{2}\right\rceil\left(\left\lceil\frac{n}{2}\right\rceil-1\right)-1,
$$
where $\xi''$ is a permutation obtained from $\xi' r^{\left\lfloor\frac{n}{2}\right\rfloor}$ by reversing $\textup{F}_2$, that is, $\xi'' = ()$.

Thus, for $m = \left\lceil\frac{\left\lfloor\frac{n}{2}\right\rfloor}{2}\right\rceil - 1$, reversing $\textup{F}_1$, shifting by two, and reversing $\textup{F}_2$ is equivalent to reflecting the permutation $s$:
\[
ss r^{\left\lceil\frac{\left\lfloor\frac{n}{2}\right\rfloor}{2}\right\rceil - 1 + \left\lfloor\frac{\left\lceil\frac{n}{2}\right\rceil}{2}\right\rfloor - 1}.
\]

Now we apply the same reflection to the permutation $s r^{n-i}$ and simplify the resulting expression:
$$
sr^{n-i}sr^{\left\lceil\frac{\left\lfloor\frac{n}{2}\right\rfloor}{2}\right\rceil - 1 + \left\lfloor\frac{\left\lceil\frac{n}{2}\right\rceil}{2}\right\rfloor - 1}
= r^{\left\lceil\frac{\left\lfloor\frac{n}{2}\right\rfloor}{2}\right\rceil - 1 + \left\lfloor\frac{\left\lceil\frac{n}{2}\right\rceil}{2}\right\rfloor - 1 - n + i}.
$$

Next, we use the Lee metric to compute the distance to the identity element $()$:
$$
\textup{dist}(r^{\left\lceil\frac{\left\lfloor\frac{n}{2}\right\rfloor}{2}\right\rceil - 1 + \left\lfloor\frac{\left\lceil\frac{n}{2}\right\rceil}{2}\right\rfloor - 1 - n + i},()) =
\textup{dist}(r^{\left\lceil\frac{\left\lfloor\frac{n}{2}\right\rfloor}{2}\right\rceil - 1 + \left\lfloor\frac{\left\lceil\frac{n}{2}\right\rceil}{2}\right\rfloor - 1 - n + i},r^0) =
$$
$$
\min \left(\ \left|\left\lceil\frac{\left\lfloor\frac{n}{2}\right\rfloor}{2}\right\rceil \!-\! 1 \!+\! \left\lfloor\frac{\left\lceil\frac{n}{2}\right\rceil}{2}\right\rfloor \!-\! 1 \!-\! n \!+\! i\right|,\ \left|n\!-\!\left(\left\lceil\frac{\left\lfloor\frac{n}{2}\right\rfloor}{2}\right\rceil \!-\! 1 \!+\! \left\lfloor\frac{\left\lceil\frac{n}{2}\right\rceil}{2}\right\rfloor \!-\! 1 \!-\! n \!+\! i\right)\right|\ \right).
$$

Case 2. Let $m = \left\lfloor\frac{\left\lfloor\frac{n}{2}\right\rfloor}{2}\right\rfloor - 1$. 
By Lemma \ref{lemma} for $j=\left\lfloor\frac{n}{2}\right\rfloor$ from the Formula \ref{lem_3} we arrive at:
$$
\textup{dist}(sr^{-\left(\left\lfloor\frac{\left\lfloor\frac{n}{2}\right\rfloor}{2}\right\rfloor - 1\right)}r^{\left\lfloor\frac{\left\lfloor\frac{n}{2}\right\rfloor}{2}\right\rfloor - 1},\xi') =
\textup{dist}(s,\xi') = \left\lfloor\frac{n}{2}\right\rfloor\left(\left\lfloor\frac{n}{2}\right\rfloor-1\right)-1,
$$
where $\xi'$ is the permutation obtained from $sr^{-\left(\left\lfloor\frac{\left\lfloor\frac{n}{2}\right\rfloor}{2}\right\rfloor - 1\right)}$ by reversing $\textup{F}_1$.

In order to use the Lemma \ref{lemma} to reverse the subsequence $\textup{F}_2$ it is necessary and sufficient to apply two shifts to the element $\xi'$.
Thus, we end up with the element $\xi'{r}^{-2}$.

Let us apply the Lemma \ref{lemma} again for $j=\left\lceil\frac{n}{2}\right\rceil$.
From the Formula \ref{lem_2} we come to:
$$
\textup{dist}(\xi' r^{-2},
\xi'' r^{\left\lceil\frac{\left\lceil\frac{n}{2}\right\rceil}{2}\right\rceil-1}) = \left\lceil\frac{n}{2}\right\rceil\left(\left\lceil\frac{n}{2}\right\rceil-1\right)-1,
$$
where $\xi''$ is the permutation obtained from $\xi' r^{-2}$ by reversing $\textup{F}_2$, that is, $\xi''$~$=$~$()$.

Thus, for $m = \left\lfloor\frac{\left\lfloor\frac{n}{2}\right\rfloor}{2}\right\rfloor - 1$, reversing $\textup{F}_1$, shifting by two, and reversing $\textup{F}_2$ is equivalent to reflecting the permutation $s$:
$$
ssr^{\left\lfloor\frac{\left\lfloor\frac{n}{2}\right\rfloor}{2}\right\rfloor - 1 + \left\lceil\frac{\left\lceil\frac{n}{2}\right\rceil}{2}\right\rceil - 1}.
$$

Now we apply the same reflection to the permutation $s r^{n-i}$ and simplify the resulting expression:
$$
sr^{n-i}sr^{\left\lfloor\frac{\left\lfloor\frac{n}{2}\right\rfloor}{2}\right\rfloor - 1 + \left\lceil\frac{\left\lceil\frac{n}{2}\right\rceil}{2}\right\rceil - 1}
= r^{\left\lfloor\frac{\left\lfloor\frac{n}{2}\right\rfloor}{2}\right\rfloor - 1 + \left\lceil\frac{\left\lceil\frac{n}{2}\right\rceil}{2}\right\rceil - 1 - n + i}.
$$

Next, we use the Lee metric to compute the distance to $()$:
$$
\textup{dist}(r^{\left\lfloor\frac{\left\lfloor\frac{n}{2}\right\rfloor}{2}\right\rfloor - 1 + \left\lceil\frac{\left\lceil\frac{n}{2}\right\rceil}{2}\right\rceil - 1 - n + i},()) =
\textup{dist}(r^{\left\lfloor\frac{\left\lfloor\frac{n}{2}\right\rfloor}{2}\right\rfloor - 1 + \left\lceil\frac{\left\lceil\frac{n}{2}\right\rceil}{2}\right\rceil - 1 - n + i},r^0) =
$$
$$
\min \left(\ \left|\left\lfloor\frac{\left\lfloor\frac{n}{2}\right\rfloor}{2}\right\rfloor \!-\! 1 \!+\! \left\lceil\frac{\left\lceil\frac{n}{2}\right\rceil}{2}\right\rceil \!-\! 1 \!-\! n \!+\! i\right|,\ \left|n\!-\!\left(\left\lfloor\frac{\left\lfloor\frac{n}{2}\right\rfloor}{2}\right\rfloor \!-\! 1 \!+\! \left\lceil\frac{\left\lceil\frac{n}{2}\right\rceil}{2}\right\rceil \!-\! 1 \!-\! n \!+\! i\right)\right|\ \right).
$$

Of the two cases considered, it is necessary to choose a decomposition with a minimum number of elements. Thus, we work out the total number of decomposition elements:
\[
\begin{aligned}
&\quad\quad\quad \left\lceil\frac{n}{2}\right\rceil \left(\left\lceil\frac{n}{2}\right\rceil-1\right)-1
 + \left\lfloor\frac{n}{2}\right\rfloor\left(\left\lfloor\frac{n}{2}\right\rfloor-1\right)-1
 + 2 \\ 
&+ \min\left(
\begin{aligned}
  &\left|\left\lceil\frac{\left\lfloor\frac{n}{2}\right\rfloor}{2}\right\rceil - 1 + \left\lfloor\frac{\left\lceil\frac{n}{2}\right\rceil}{2}\right\rfloor - 1 - n + i\right|,\\
  &\left|n-\left(\left\lceil\frac{\left\lfloor\frac{n}{2}\right\rfloor}{2}\right\rceil - 1 + \left\lfloor\frac{\left\lceil\frac{n}{2}\right\rceil}{2}\right\rfloor - 1 - n + i\right)\right|,\\
  &\left|\left\lfloor\frac{\left\lfloor\frac{n}{2}\right\rfloor}{2}\right\rfloor - 1 + \left\lceil\frac{\left\lceil\frac{n}{2}\right\rceil}{2}\right\rceil - 1 - n + i\right|,\\
  &\left|n-\left(\left\lceil\frac{\left\lceil\frac{n}{2}\right\rceil}{2}\right\rceil - 1 + \left\lfloor\frac{\left\lfloor\frac{n}{2}\right\rfloor}{2}\right\rfloor - 1 - n + i\right)\right|
\end{aligned}
\right).
\end{aligned}
\]

We now simplify the resulting formula. Table~\ref{tab:mod4_params} lists the values of the floor and ceiling functions depending on the residue class of $n$ modulo~$4$.

\begin{table}[h]
\centering
\renewcommand{\arraystretch}{1.6}
\setlength{\tabcolsep}{10pt}
\addtolength{\extrarowheight}{5pt} % чтобы дроби не прилипали к линиям

\begin{tabular}{l|ccccc}
 & $\left\lfloor\frac{\left\lfloor\frac{n}{2}\right\rfloor}{2}\right\rfloor$ & $\left\lceil\frac{\left\lceil\frac{n}{2}\right\rceil}{2}\right\rceil$& $\left\lfloor\frac{\left\lceil\frac{n}{2}\right\rceil}{2}\right\rfloor$ & $\left\lceil\frac{\left\lfloor\frac{n}{2}\right\rfloor}{2}\right\rceil$ & $\left\lfloor
\frac{n}{2}
\vphantom{\frac{\left\lceil\frac{n}{2}\right\rceil}{2}}
\right\rfloor$  \\
\midrule
$n=4k$   & $k$ & $k$  & $k$   & $k$    & $2k$ \\
$n=4k+1$ & $k$ & $k+1$& $k$   & $k$    & $2k$ \\
$n=4k+2$ & $k$ & $k+1$& $k+1$ & $k$    & $2k+1$ \\
$n=4k+3$ & $k$ & $k+1$& $k+1$ & $k+1$  & $2k+1$ \\
\end{tabular}
\selectlanguage{english}
\caption{Values of floor and ceiling functions}
\label{tab:mod4_params}
\end{table}

From the table, it follows that $\min\left( \left\lfloor\frac{\left\lfloor\frac{n}{2}\right\rfloor}{2}\right\rfloor+\left\lceil\frac{\left\lceil\frac{n}{2}\right\rceil}{2}\right\rceil,\left\lfloor\frac{\left\lceil\frac{n}{2}\right\rceil}{2}\right\rfloor + \left\lceil\frac{\left\lfloor\frac{n}{2}\right\rfloor}{2}\right\rceil \right)$ equals $2k$ for the first two cases and $2k+1$ for the remaining ones, and that this minimum is equal to $\left\lfloor\frac{n}{2}\right\rfloor$.
Thus, we obtain the expression:
\[
 \left\lceil\frac{n}{2}\right\rceil \left(\left\lceil\frac{n}{2}\right\rceil-1\right)
 + \left\lfloor\frac{n}{2}\right\rfloor\left(\left\lfloor\frac{n}{2}\right\rfloor-1\right)
  + \min\left(
  \left|\left\lfloor\frac{n}{2}\right\rfloor - 2 - n + i\right|,
  \left|n-\left(\left\lfloor\frac{n}{2}\right\rfloor - 2 - n + i\right)\right|
\right).
\]

Substituting different ranges of $i$, we obtain the desired estimate.
The theorem is proved.

\end{proof}

\begin{theorem}
A lower bound for the diameter of Cayley graph of the symmetric group $S_n$ generated by  $(12), (12 \dots n), (1n \dots 2)$ is $\frac{n(n-1)}{2}$ for $n \geqslant 4$.
\end{theorem}

\begin{proof}
\normalfont
The lower bound on the graph diameter follows from Theorem~\ref{th_1} for $i=2$:
$$
\textup{dist}(\ sr^{n-2},()\ ) = \left\lceil\frac{n}{2}\right\rceil \left(\left\lceil\frac{n}{2}\right\rceil-1\right)-1 + \left\lfloor\frac{n}{2}\right\rfloor\left(\left\lfloor\frac{n}{2}\right\rfloor-1\right)-1 +
\left\lfloor\frac{n}{2}\right\rfloor - 2 + 4 =
$$
$$
\left\lceil\frac{n}{2}\right\rceil^2 + \left\lfloor\frac{n}{2}\right\rfloor^2 - \left\lceil\frac{n}{2}\right\rceil.
$$
For $n=2z$, we obtain $z^2+z^2-z = z(2z-1)=\frac{n(n-1)}{2}$.
For $n=2z+1$, we obtain $(z+1)^2+z^2-(z+1) = z(2z+1)=\frac{n(n-1)}{2}$.
This completes the proof of the theorem.
\end{proof}

\textbf{The author expresses gratitude to Alexander Chervov} (Institut Curie, Paris) for posing the problem and for the invitation to participate in the CayleyPy project, \textbf{as well as to Kirill Andreevich Popkov} (Keldysh Institute of Applied Mathematics, RAS, Moscow) for valuable remarks.

%=================References====================
\end{fulltext}

\end{document}